\documentclass[reqno]{amsart}
\usepackage{amsmath, amsthm, amscd, amsfonts, amssymb, graphicx, color}
\usepackage[bookmarksnumbered,plainpages]{hyperref}

\newtheorem{thm}{Theorem}[section]
\newtheorem{lem}[thm]{Lemma}
\newtheorem{prop}[thm]{Proposition}
\newtheorem{cor}[thm]{Corollary}
\theoremstyle{defn}

\theoremstyle{rem}

\numberwithin{equation}{section}
\def\bit{\begin{itemize}}
\def\eit{\end{itemize}}
\def\beq*{\begin{eqnarray*}}
\def\eeq*{\end{eqnarray*}}
\newcommand{\G}{\mathbb{G}}

\newcommand{\CL}{{\rm C}_0(\G)}

\newcommand{\MBL}[1]{\mathcal{B}_{l}^{\sigma}(#1)}
\newcommand{\MBR}[1]{\mathcal{B}_{r}^{\sigma}(#1)}
\newcommand{\MG}{{\rm M}(\G)}
\newcommand{\LL}{{\rm L}^1(\G)}
\newcommand{\LG}{{\rm L}^\infty(\G)}
\newcommand{\WAP}{{\rm WAP}(\G)}
\newcommand{\LUC}{{\rm LUC}(\G)}
\newcommand{\RUC}{{\rm RUC}(\G)}

\newcommand{\M}{\mathfrak{M}}
\newcommand{\vt}{\bar{\otimes}}
\newcommand{\no}[1]{\|#1\|}
\newcommand{\fs}{\hspace{0.3cm}}

\begin{document}
\title[Module homomorphisms and multipliers]{Module homomorphisms and multipliers on locally compact quantum groups }
\author{M. Ramezanpour}
\address{Department of Pure Mathematics, Ferdowsi University of Mashhad, P. O. Box 1159, Mashhad 91775, Iran.}
\email{mo\_\ ra54@stu-math.um.ac.ir}
\author{H.\ R.\ E.\ Vishki }
\address{Department of Pure Mathematics and Centre of Excellence
in Analysis on Algebraic Structures (CEAAS), Ferdowsi University of Mashhad\\
P. O. Box 1159, Mashhad 91775, Iran.} \email{vishki@ferdowsi.um.ac.ir}

\subjclass[2000]{22D15, 22D25, 43A22, 46H25, 46L65, 46L89, 81R50}
\keywords{Locally compact quantum group, module homomorphism, Wendel's theorem, Hopf-von Neumann algebra, multiplier, topological centre}
\maketitle
\centerline{\it (A tribute to Professor M.A. Pourabdollah on the occasion of his 65th birthday)}
\begin{abstract}
For a Banach algebra $A$ with a bounded approximate identity, we investigate  the  $A$-module homomorphisms of certain introverted subspaces of $A^*$,  and  show that  all  $A$-module homomorphisms of $A^*$ are  normal if and only if $A$ is an ideal of $A^{**}$. We obtain some characterizations of compactness and discreteness for a locally compact quantum group $\G$. Furthermore, in the co-amenable case we prove that the multiplier algebra of $\LL$ can be identified with $\MG.$ As a consequence, we prove that $\G$ is compact if and only if $\LUC={\rm WAP}(\G)$ and $\MG\cong\mathcal{Z}({\rm LUC}(\G)^*)$; which partially answer a problem raised by Volker Runde.
\end{abstract}

\section{ Introduction}
In the realm of Hopf-von Neumann algebras, the concept of locally compact quantum group was first introduced by Kustermans and Vaes in \cite{K-V} and \cite{K-V.von}. They used a set of comprehensive axioms in their definition  which cover the notion of Kac algebras, ~\cite{E-S}, and thus all locally compact groups. Some of the most famous locally compact quantum groups---which have  been extensively studied in abstract harmonic analysis---are  ${\rm L}^\infty(G)$ and ${\rm VN}(G)$ which, in some sense, are the dual of each others. In spite of  these two algebras have different  features in abstract harmonic analysis but they have mostly a unified framework from the locally compact quantum group point of view.

In  \cite{Run} and  \cite{R_uniform}   Runde  reformulated  some existing results in abstract harmonic analysis---related to the group algebra ${\rm L}^1(G)$ and the Fourier algebra ${\rm A}(G)$---in the language of locally compact quantum groups. For instance, he showed that for a locally compact quantum group $\G$, the algebra $\LL$ is an ideal in $\LL^{**}$ if and only if $\G$ is compact, \cite[Theorem 3.8]{Run}. The main theme of the present paper is to  unify some other existing results in abstract harmonic analysis---related to the  module homomorphisms, multipliers and topological centre ---in the language of locally compact quantum groups.

The paper is organized as follows. Section~\ref{mh} deals with the  module homomorphisms. For a left introverted subspace $X$ of $A^*$, in which, $A$ is a Banach algebra enjoying an approximate identity bounded by one, it is proved that if $X\subseteq A^*\cdot A$,  then the algebra of all right $A$-module homomorphisms can be identified with $X^*;$ from which we characterize  the normal (i.e. w*-w*-continuous) $A$-module homomorphisms of $A^*$. In particular, we show that a co-amenable locally compact quantum group $\G$ is compact if and only if  all $\LL$-module homomorphisms of $\LG$ are normal; and that $\G$ (not necessarily co-amenable) is discrete if and only if all $\CL$-module homomorphisms of $\MG$ are normal. The results of this section cover many results of ~\cite{L_oper} and \cite{L_uni}.

In Sections~\ref{multi} we show that for a co-amenable locally compact quantum group $\G$,  the left (and also right)  multiplier algebras  of $\LL$ can be identified with $\MG.$ This extends both classical results of ~\cite{Derigh, Wend} about the multiplier algebras of ${\rm L}^1(G)$ and ${\rm A}(G)$. Some interrelations between the multipliers and the operators which commute with translations are also investigated. As a consequence some results of ~\cite{B-E} are extended.

Finally in section \ref{embed}, we show that if $\G$ is a co-amenable locally compact quantum group, then there exists an isometric homomorphism $\Theta$ from $\MG$ into $\LUC^*$ such that $\Theta$ is onto if and only if $~\G$ is compact. We use this to partially answer a problem raised by  Runde, \cite{R_uniform}: if $~\G$ is co-amenable then $\G$ is compact if and only if $\LUC={\rm WAP}(\G)$ and $\Theta(\MG)=\mathcal{Z}(\LUC^*)$. We show that some of the main results of \cite{L_cont} and \cite{L-LOS} are unified by the results of this section.

\section{Left introverted subspaces and module homomorphisms }\label{mh}
Let $A$ be a Banach algebra.
For $\omega\in {A}$ and $x\in {A}^*$, we define the elements $x\cdot\omega$ and $\omega\cdot x$ of ${A}^*$
by the formula
$$ x\cdot\omega(\nu)= x(\omega\nu)\fs{\rm and}\fs \omega\cdot x(\nu)=x(\nu\omega)\fs(\nu\in A).$$
These are the natural module actions of $A$ on $A^*$ which turns $A^*$ into a Banach $A$-module.
A closed subspace $X$ of $A^*$ is called left invariant  if
$X\cdot A\subseteq X$, where $X\cdot A=\{x\cdot\omega : x\in X, \omega\in A\}$.
A left invariant  subspace $X$ is called  left introverted if $X^*\odot X\subseteq X$; in which, the elements $m\odot x$ of $A^*$ are defined so that
$$ m\odot x(\omega)=m(x\cdot\omega)\fs\fs (\omega\in A, m\in X^* , x\in X).$$
A right introverted subspace of $A^*$ can be defined similarly.
Let $X$ be a left introverted subspace of $A^*$, then the following   Arens (type)  product on $X^*$
$$(m\odot n)(x)=m(n\odot x)\qquad (m, n\in X^*, x\in X),$$
turns it into a Banach algebra.  We define the   topological centre $\mathcal{Z}(X^*)$ of $X^*$ by
 $$\mathcal{Z}(X^*):=\{~n\in X^*\fs: ~m\mapsto n\odot m \fs {\rm is~ normal~ on~ } X^*~\}.$$
clearly the product $\odot$ on $X=A^*$ gives the so-called first Arens product on $A^{**}$.

We use $\langle A^*\cdot A\rangle$ and $\langle A\cdot A^*\rangle$ to denote the closed linear spans of $A^*\cdot A$ and $A\cdot A^*$, respectively. It can be readily verified that $\langle A^*\cdot A\rangle$ and $\langle A\cdot A^*\rangle$ are  left and  right introverted in $A^*$, respectively. In the case where $A$ has a bounded approximate identity (=BAI), Cohen Factorization Theorem, \cite[Corollary 2.9.26]{Dal00}, implies that $\langle A^*\cdot A\rangle=A^*\cdot A$ and $\langle A\cdot A^*\rangle=A\cdot A^*$. It is known that the weakly almost periodic elements ${\rm WAP}(A)$ of $A^*$, is also an  introverted  subspace of $A^*$.

Let $\G=(\M,\Gamma,\varphi,\psi)$ be a von Neumann  algebraic locally compact quantum group as described in ~\cite{K-V, K-V.von}. By definition, the pair $(\M,\Gamma)$ is a Hopf-von Neumann algebra and $\varphi, \psi$ are left and right invariant normal semifinite faithful weights on $\M$, respectively.
Since $\Gamma:\M\to\M\vt\M$  is a normal, unital $*$-homomorphism
 satisfying $(\iota\otimes\Gamma)\Gamma=(\Gamma\otimes\iota)\Gamma$, it is well known that its pre-adjoint induces an associative multiplication  $\ast$ on  $\M_*$. Here $\vt$ denotes the von Neumann algebra tensor product. It is worthwhile mentioning that in the two classical case $\G_a=({\rm L}^\infty(G),\Gamma_a,\varphi_a,\psi_a)$ and
$\G_s=({\rm VN}(G),\Gamma_s,\varphi_s,\psi_s )$, \cite{E-S}, the product $\ast$ imposed on the preduals is just the usual convolution on ${\rm L}^1(G)$ and the pointwise multiplication on ${\rm A}(G)$, respectively.

Analogue to the locally compact group case, for a locally compact quantum group $\G$ we write $\LG$ for $\M,$ $\LL$ for $\M_*,$ $\LUC$ for $\langle\LG\cdot\LL\rangle$, $\RUC$ for $\langle\LL\cdot\LG\rangle$ and $\WAP$ for ${\rm WAP}(\LL)$. Note that, in this case  the $\LL$-module actions of $\LG$ can be presented  by
$$\omega\cdot x=(\iota\otimes\omega)(\Gamma x) \ \ {\rm and}\ \ x\cdot\omega =(\omega\otimes \iota)(\Gamma x) \qquad (\omega\in\LL, x\in\LG).$$

Let $\G$ be a locally compact quantum group and $(\CL,\Gamma_c,\varphi_c,\psi_c)$ be the reduced $C^*$-algebraic quantum group of $\G$ as \cite[Proposition 1.7]{K-V.von}. We also write $\MG$ for the  dual $\CL^*$ of $\CL$. Then $\MG$ is a Banach algebra under the product
$\ast$ given by
$(m\ast n)(x)=(m\otimes n)\Gamma_c(x)$\ ($m,n\in \MG, x\in\CL$).

Examining these algebras for the classical cases $\G_a$ and $\G_s$ yield some well-known algebras on the group $G$. More precisely, ${\rm L}^\infty(\G_a)$, ${\rm L}^1(\G_a)$, ${\rm LUC}(\G_a)$, ${\rm RUC}(\G_a)$, ${\rm WAP}(\G_a)$, ${\rm C}_0(\G_a)$ and ${\rm M}(\G_a)$ coincide with ${\rm L}^\infty(G)$, ${\rm L}^1(G)$, ${\rm LUC}(G)$, ${\rm RUC}(G)$, ${\rm WAP}(G)$, ${\rm C}_0(G)$ and  ${\rm M}(G)$, respectively, as explicitly introduced in \cite{H-R}. While ${\rm L}^\infty(\G_s)$, ${\rm L}^1(\G_s)$, ${\rm LUC}(\G_s)={\rm RUC}(\G_s)$, ${\rm WAP}(\G_s)$,   ${\rm C}_0(\G_s)$ and ${\rm M}(\G_s)$ coincide with  ${\rm VN}(G)$, ${\rm A}(G)$, $UCB(\hat{G})$, $W(\hat{G})$, ${\rm C}_\rho^*(G)$ and $ {\rm B}_\rho(G)$, respectively, as described in \cite{Eym, L_uni}.\\

For a locally compact group $G$, it is obvious that the Arens product on ${\rm C}_0(G)^*$ and the usual convolution product on ${\rm M}(G)$  coincide. Also   as it is shown that in \cite[Proposition 5.3]{L_uni},  the Arens product on ${\rm C}_\rho^*(G)^*={\rm B}_\rho(G)$ is precisely the pointwise product on it . Following we show that the same is also true for a general locally compact quantum group $\G$.

\begin{prop}\label{MG}
 For every locally compact quantum group $\G$, the products $\odot$ and $\ast$ on $\MG$ coincide.
\end{prop}

\begin{proof}
For every  $m,n\in \MG$ and $x\in\CL$ we have $(m\ast n)(x)=m((\iota\otimes n)\Gamma_c (x))=n((m\otimes\iota)\Gamma_c(x)).$
Since  $(\iota\otimes n)\Gamma_c(x)$ and  $(m\otimes\iota)\Gamma_c(x)$ lie in  $\CL$, \cite[Section 4]{K-V}, the product $*$ on $\MG$ is w*-separate continuous. A direct verification shows that $\CL$ is an invariant subspace of $\LG$  and so the inclusion $\CL\subseteq \WAP$, \cite[Theorem 4.3]{R_uniform}, implies that $\CL$ is also an introverted subspace of $\LG$. The product $\odot$ on  $\CL^*=\MG$ is w*-separately continuous.  Since $\ast=\odot$ on  $\LL$ and  $\LL$  is w*-dense in $\MG$ we conclude that $\ast=\odot$ on the whole of $\MG$.
\end{proof}
\vspace{.5cm}

 Let $X$ be an invariant subspace of $A^* $, then clearly $X$ is a Banach  $A$-module. We denote by $\mathcal{H}_r(X)$ and  $\mathcal{H}_\ell(X)$ the space of all left and right $A$-module homomorphisms of $X$; that is
\begin{eqnarray*}
&&\mathcal{H}_r(X)=\{~T\in\mathcal{B}(X) : T(x\cdot\omega)=T(x)\cdot\omega \ \ {\rm for \ all}\  x\in X , \omega\in A\},\ {\rm and}\\
&&\mathcal{H}_\ell(X)=\{~T\in\mathcal{B}(X) : T(\omega\cdot x)=\omega\cdot T(x) \ \ {\rm for \ all} \ x\in X , \omega\in A\}.
\end{eqnarray*}
Here ${\mathcal B}(X)$ denote the Banach space of all bounded operators on $X$. By $\mathcal{H}_r^\sigma(X)$ and $\mathcal{H}_\ell^\sigma(X)$ we mean  all normal elements of
$\mathcal{H}_r(X)$ and  $\mathcal{H}_\ell(X)$, respectively.\\

We  now give a characterization of $\mathcal{H}_r(X)$ and  $\mathcal{H}_\ell(X)$ in terms of the dual of  certain subspaces of $A^*\cdot A$ and $A\cdot A^*$, respectively.

\begin{prop}\label{XA=HL}
Let $A$ be a Banach algebra with an approximate identity bounded by one, then
\bit
\item[(i)] If $X$ is a left introverted subspace of $A^*$ then  $(X\cdot A)^*\cong \mathcal{H}_r(X)$. More precisely, the mapping $m\mapsto L_m :(X\cdot A)^*\to \mathcal{H}_r(X)$ is an isometric (algebraic) isomorphism,
where $(L_m(x))(\omega)=m(x\cdot\omega)$ for $x\in X, \omega\in A$.
\item[(ii)] If $X$ is a right introverted subspace of $A^*$ then $(A\cdot X)^*\cong \mathcal{H}_\ell(X)$.
\eit
\end{prop}

\begin{proof} (i) A direct verification shows that $X\cdot A$ is  a left introverted subspace of $A^*$,
$L_m\in \mathcal{H}_r(X)$ and
the mapping $m\mapsto L_m$ is a contractive  homomorphism.
Let $\{\omega_\alpha\}$ be a BAI  in $A$ with  $\no{\omega_\alpha}\leq 1$, then  $\no{z\cdot \omega_\alpha-z}\to 0$, for each $z\in (X\cdot A),$ and so
$$\no{L_m(z)}\geq |L_m(z)(\omega_\alpha)|=|m(z\cdot \omega_\alpha)|\to |m(z)|,$$
from  which we get $\no{L_m}\geq \no{m}$. Let $T\in \mathcal{H}_r(X)$, then $T(X\cdot A)\subseteq (X\cdot A)$. Embed $A$ into $(X\cdot A)^*$ in the natural way and let $m$ be a w*-cluster point of the net $\{T^*(\omega_\alpha)\}$ in $(X\cdot A)^*$. Then for every $x\in X$ and $\omega\in A$ we have
\begin{eqnarray*}
T(x)(\omega)=\lim_\alpha T(x)(\omega*\omega_\alpha)=\lim_\alpha (T(x)\cdot\omega)(\omega_\alpha)
=\lim_\alpha T(x\cdot\omega)(\omega_\alpha)=m(x\cdot\omega)
=L_{m}(x)(\omega);
\end{eqnarray*}
implying that, $m\mapsto L_m$ is surjective. A similar argument may apply to prove (ii).
\end{proof}

As a consequence of Proposition \ref{XA=HL} we have the next result,  part (iii) of which has already proved in \cite[Proposition 4.1]{Neu}.

\begin{cor}\label{A*A}
Let $A$ be Banach algebra with an approximate identity bounded by one, then
\bit
\item[(i)] If $X$ is a left introverted subspace of  $A^*\cdot A$ then $\mathcal{H}_r(X)\cong X^*$;
\item[(ii)] If $X$ is a right introverted subspace of $A\cdot A^*$ then $\mathcal{H}_\ell(X)\cong X^*$;
\item[(iii)] $\mathcal{H}_r(A^*)\cong (A^*\cdot A)^*\cong\mathcal{H}_r(A^*\cdot A)$ and
$\mathcal{H}_\ell(A^*)\cong (A\cdot A^*)^*\cong\mathcal{H}_\ell(A\cdot A^*).$
\eit
\end{cor}

\begin{proof}
(i) It suffices  to show that $X\cdot A=X$. It follows from $X\subseteq A^*\cdot A$, that $\no{x\cdot \omega_\alpha-x}\to 0$, foe all $x\in X$, where $\{\omega_\alpha\}$ is a BAI of $A$. This and the fact that  $X\cdot A$ is closed imply that $X\cdot A=X$. Part (ii) can be proved in a similar fashion. For (iii) it is enough to apply parts (i) and (ii), and also apply  Proposition \ref{XA=HL} for $X=A^*.$
\end{proof}

A locally compact quantum group $\G$ is called co-amenable if $\LL$ has a BAI, \cite{B-T}. In \cite[Theorem 2]{HNR}, it is shown that $\G$ is co-amenable if and only if $\LL$ has a BAI consisting of normal states of $\LG.$ Therefore,  in the two classical cases $\G_a$ and $\G_s$;  $\G_a$ is always co-amenable and  $\G_s$ is co-amenable if and only if  the group $G$ is amenable.

In the next result which  simultaneously extend \cite[Theorem 3]{L_oper} and \cite[Corollaries 6.4, 6.5]{L_uni}, we  examine the conclusion of Corollary~\ref{A*A} for the case that $A=\LL$ and $X=\CL$.

\begin{thm}\label{LG}
If $\G$ is a co-amenable locally compact quantum group, then
\bit
\item[(i)] $\mathcal{H}_r(\CL)\cong \MG\cong\mathcal{H}_\ell(\CL)$;
\item[(ii)] $\mathcal{H}_r(\LG)\cong \LUC^*\cong \mathcal{H}_r(\LUC)$ and
$\mathcal{H}_\ell(\LG)\cong {\rm RUC}(\G)^*\cong \mathcal{H}_\ell({\rm RUC}(\G)).$
\eit
\end{thm}

The next result investigate the normal elements of $\mathcal{H}_r(X)$. Note that, for every left (or right) introverted subspace $X$ of $A^*$ one may embed $A$ into $X^*$ in a natural way.

\begin{prop}\label{HL=HLS}
Let $A$ be a Banach algebra with a BAI, and let  $X$ be a left introverted subspace of $A^*$. Consider the following statement:
\bit
\item[(i)] $A\odot X^*\subseteq A$;
\item[(ii)] $\mathcal{H}_r(X)=\mathcal{H}_r^\sigma(X)$;
\item[(iii)] Every element of $A\odot X^*$ is normal on $X$.
\eit
Then (i)$\Rightarrow$(ii)$\Rightarrow$(iii) and these are equivalent if $A^*\cdot A\subseteq X.$
\end{prop}

\begin{proof}
 (i)$\Rightarrow$(ii): Let $T\in\mathcal{H}_r(X)$. As $A^2=A$, it is enough to show that
$T^*(A^2)\subseteq A$. To see this, let $\omega, \nu\in A$ and $x\in X,$
then
\begin{eqnarray*}
T^*(\omega\nu)(x)=(T(x)\cdot\omega)(\nu)
=T(x\cdot\omega)(\nu)
=(\omega\odot T^*(\nu))(x),
\end{eqnarray*}
where, $T^*(\omega\nu)=\omega\odot T^*(\nu)\in A\odot X^*\subseteq A$, as required.\\
(ii)$\Rightarrow$(iii): Take $m\in X^*, \omega\in A$ and let $\pi:X^*\to (X\cdot A)^*$ be  the adjoint of the natural embedding $(X\cdot A)\hookrightarrow X$. Then by Proposition \ref{XA=HL}, $L_{\pi(m)}\in\mathcal{H}_r(X)$ and so it is normal. Let $\{x_\alpha\}\subseteq X$ w*-converges to $x\in X$, then
\begin{eqnarray*}
(\omega\odot m)(x_\alpha)=m(x_\alpha\cdot\omega)
=L_{\pi(m)}(x_\alpha)(\omega)
\to L_{\pi(m)}(x)(\omega)=(\omega\odot m)(x).
\end{eqnarray*}
Hence $\omega\odot m$ is normal on $X$.\\
 Now we prove (iii)$\Rightarrow$(i) under the additional assumption  $A^*\cdot A\subseteq X$. Let $m\in X^*,$ $\omega, \nu\in A$ and let $\tilde{m}\in A^{**}$ be an extension of $m$. Then for every net  $\{x_\alpha\}\subseteq A^*$ which w*-converges  to $x\in A^*$, we have $$(\omega\nu\odot\tilde{m})(x_\alpha)=(\nu\odot\tilde{m})(x_\alpha\cdot\omega)=(\nu\odot m)(x_\alpha\cdot\omega)\to(\nu\odot m)(x\cdot\omega)=(\omega\nu\odot\tilde{m})(x).$$
Hence $\omega\nu\odot \tilde{m}\in A.$ This together with the fact that $A^2=A$ imply that $\omega\odot m\in A$.
\end{proof}

For a locally compact quantum group $\G$, it is known that $\LL$ is an ideal of $\MG$; see \cite[pp. 193-194]{K-V}. In particular, $\LL\odot\CL^*=\LL\ast \MG\subseteq\LL$ by Proposition \ref{MG}. Using Proposition \ref{HL=HLS} we obtain $\mathcal{H}_r(\CL)=\mathcal{H}_r^\sigma(\CL)$. Combine this fact with the part (i) of Theorem~\ref{LG} we get the next corollary.

\begin{cor}
If $\G$ is a co-amenable locally compact quantum group, then
\[\mathcal{H}_r^\sigma(\CL)=\mathcal{H}_r(\CL)\cong {\rm M}(\G)\cong\mathcal{H}_\ell(\CL)=\mathcal{H}_\ell^\sigma(\CL).\]
\end{cor}

 Using Proposition \ref{HL=HLS} for $X=A^*$ we get the next result.

\begin{cor}\label{id}
Let $A$ be  a Banach algebra   with a BAI, then
\bit
\item[(i)]  $\mathcal{H}_r(A^*)=\mathcal{H}_r^\sigma(A^*)$ if and only if $A$ is a right ideal of $A^{**}$;
\item[(ii)]  $\mathcal{H}_\ell(A^*)=\mathcal{H}_\ell^\sigma(A^*)$ if and only if $A$ is a left ideal of $A^{**}$.
\eit
\end{cor}

 Recall that a locally compact quantum group $\G$ is said to be compact if its Haar weights are finite, or equivalently, $\CL$ is unital. $\G$ is said to be discrete if its dual quantum group $\hat{\G}$ is compact, or equivalently, $\LL$ is unital; for more details see~\cite{Run}. For a locally compact quantum group $\G$, it was shown in \cite[Theorem 3.8]{Run} that $\LL$ is an ideal of $\LL^{**}$ if and only $\G$ is compact. Utilizing this fact and Corollary~\ref{id} for $A=\LL$, we obtain the next result, which is an extension of \cite[Theorem 2]{L_oper}. It seems that the next result for the classical case $\G_s$ has already been known in the context of Fourier algebras; however we don't know a reference for it.

\begin{thm}\label{Comp}
For every co-amenable locally compact quantum group $\G$, the following assertions are equivalent:
\bit
\item[(i)] $\mathcal{H}_r(\LG)=\mathcal{H}_r^\sigma(\LG)$;
\item[(ii)] $\mathcal{H}_\ell(\LG)=\mathcal{H}_\ell^\sigma(\LG)$;
\item[(iii)] $\G$ is compact.
\eit
\end{thm}

\begin{proof}
Based on what we have mentioned just before Theorem \ref{Comp}, it is enough to show that  $\LL$ is a left ideal in $\LL^{**}$ if and only if it is a right ideal in $\LL^{**}.$ To show this, let $\kappa$ be the unitary antipode of $\G$ and let $\LL$ be a right ideal in $\LL^{**}$. Then for each $x\in\LG$, and each $\omega\in\LL$ we have
\begin{eqnarray*}
\kappa(x)\cdot\kappa^*(\omega)=(\omega\circ\kappa\otimes\iota)\Gamma(\kappa(x))
=(\iota\otimes\omega\circ\kappa)(\kappa\otimes\kappa)\Gamma(x)
=\kappa((\iota\otimes\omega)\Gamma(x))
\end{eqnarray*}
and so $m\cdot \omega=\kappa^*\big{(}\kappa^*(\omega)\cdot \kappa^*(m)\big{)}\in\LL$ for all $m\in\LL^{**}$. Thus, $\LL$ is a left ideal of $\LL^{**}$, as claimed.
\end{proof}

A result of Runde \cite[Theorem 4.4]{Run} states that $\CL$ is an ideal of $\CL^{**}$ if and only if $\G$ is discrete. We use this fact in the next result.

\begin{thm}\label{Disc}
For every  locally compact quantum group $\G$, the following assertions are equivalent:
\bit
\item[(i)] $\mathcal{H}_r(\MG)=\mathcal{H}_r^\sigma(\MG)$;
\item[(ii)] $\mathcal{H}_\ell(\MG)=\mathcal{H}_\ell^\sigma(\MG)$;
\item[(iii)] $\G$ is discrete.
\eit
\end{thm}

\begin{proof}
Let $\kappa_c$ be the unitary antipode of the reduced C$^*$-algebraic quantum group $(\CL,\Gamma_c,\varphi_c,\psi_c)$ and let $\CL$ be a right ideal in $\CL^{**}$. Then for every  $x\in\CL$ and $m\in\MG$ we have
\begin{eqnarray*}
\kappa_c^*(x\cdot m)(y)=(x\cdot m)(\kappa_c(y))=m(\kappa_c(y)x)=\kappa_c^*(m)(\kappa_c(x)y)
=(\kappa_c^*(m)\cdot \kappa_c(x))(y).
\end{eqnarray*}
Thus  $f\cdot x=\kappa_c\big{(}\kappa_c(x)\cdot \kappa_c^{**}(f)\big{)}\in\CL$, for every $f\in\CL^{**}$. In other words, $\CL$ is a left ideal of $\CL^{**}$. Now apply Corollary \ref{id} for $A=\CL$.
\end{proof}

\section{The  Multiplier algebra of $\LL$ }\label{multi}
For a Banach algebra $A$, let  $\mathcal{RM}(A)$ and $\mathcal{LM}(A)$ be the right and left
multiplier algebra of $A$, respectively; that is
\begin{eqnarray*}
&&\mathcal{LM}(A)=\{~T\in \mathcal{B}(A)~:~T(ab)=T(a)b\quad {\rm for all}~~~ a,b\in A~\},\ {\rm and }\\
&&\mathcal{RM}(A)=\{~T\in \mathcal{B}(A)~:~T(ab)=aT(b)\quad {\rm for all}~~~ a,b\in A~\}.
\end{eqnarray*}

A famous result of Wendel, \cite{Wend}, states that for every locally compact group $G$, the left (and right) multiplier algebras of ${\rm L}^1(G)$ can be identified with ${\rm M}(G)$; that is, $\mathcal{LM}({\rm L}^1(G))\cong {\rm M}(G)\cong \mathcal{RM}({\rm L}^1(G)).$ For the Fourier algebra $A(G)$, it has been  shown (see \cite{Derigh}, for instance) that, in the case where $G$ is amenable, then  $\mathcal{LM}({\rm A}(G))=\mathcal{RM}({\rm A}(G))\cong {\rm B}(G).$ Here we extend both of these by a general result related to locally compact quantum groups.

\begin{thm}\label{LM=RM}
For every co-amenable locally compact quantum group $\G$,
\[\mathcal{LM}(\LL)\cong \MG\cong \mathcal{RM}(\LL).\]
More precisely, the mapping $m\mapsto R_m : \MG\rightarrow  \mathcal{RM}(\LL)$ is an isometric isomorphism, where $R_m(\omega)=\omega \ast m$ for $\omega\in\LL$.
\end{thm}

\begin{proof}
A direct verification shows that, for every  $m, n\in \MG$
$$R_m\in \mathcal{RM}(\LL)\ ,\ \  \|R_m\|\leq\|m\|\ \  {\rm and }\ \  R_{m\ast n}=R_m R_n.$$
Suppose that $T\in \mathcal{RM}(\LL)$. It is enough to show that
there exists an element $m\in \MG$ satisfying $\no{m}\leq \no{T}$ and $T=R_m$. Let $\{\omega_\alpha\}$
be an approximate identity bounded by one for $\LL$.
Then $\{T(\omega_\alpha)\}$ is a net in $\LL\subseteq \MG$ bounded by  $\no{T}$.
By passing to a subnet, if necessary,  we may assume that $T(\omega_\alpha)$ w*-converges to some  $m\in \MG$ with $\no{m}\leq\no{T}.$ By Proposition \ref{MG} the product $*$ on  $\MG$ is w*-separately continuous. So for $\omega\in\LL$ and $x\in \CL$ we have
\begin{eqnarray*}
T(\omega)(x)= \lim_\alpha T(\omega*\omega_\alpha)(x)=\lim_\alpha\ (\omega*T(\omega_\alpha))(x)
=(\omega\ast m)(x)=R_m(\omega)(x).
\end{eqnarray*}
Since $\CL$ is w*-dense in $\LG$, we have $T=R_m$ with $\no{m}\leq\no{T}$.
\end{proof}

Let $(\M,\Gamma)$ be a Hopf-von Neumann algebra and $\phi\in\mathcal{B}(\M)$. We say that $\phi$ commutes with left (resp. right)
translations if~~
$\Gamma\phi=(\iota\otimes\phi)\Gamma~~({\rm resp}.~\Gamma\phi=(\phi\otimes\iota)\Gamma)$.
We denote by $\MBL{\M}~({\rm resp}.~~\MBR{\M})$ the space of all  bounded
normal linear mappings on $\M$ commuting with the left (resp. right) translations. For the classical  case  $({\rm L}^\infty(G),\Gamma_a)$ one may easily verify that a bounded normal linear mapping $\phi$ on ${\rm L}^\infty(G)$ lies in $\MBL{{\rm L}^\infty(G)}~({\rm resp}.~~\MBR{{\rm L}^\infty(G)})$ if and only if
$\phi(l_sf)=l_s\phi(f)~~({\rm resp}.~\phi(r_sf)=r_s\phi(f))$ for each $s\in G$ and $f\in {\rm L}^\infty(G)$, where $l_sf(t)=f(st)=r_tf(s).$\\

It is well-known that $\mathcal{RM}(A)\cong\mathcal{H}_r^\sigma(A^*).$ Indeed, a direct verification shows that
$T\mapsto T^*: \mathcal{RM}(A)\rightarrow\mathcal{H}_r^\sigma(A^*)$ is an isometric algebra isomorphism. Similarly, $\mathcal{LM}(A)\cong\mathcal{H}_\ell^\sigma(A^*).$ Applying these facts for $A=\LL$, as a combination of  Theorem~\ref{LG} and Theorem~\ref{LM=RM}, we have the next result, that covers a series of  results presented in  \cite{B-E} and \cite{L_uni}; see \cite[Theorem 6.6]{L_uni}, for instance.

\begin{cor}\label{MBR}
For every co-amenable locally compact quantum group $\G$,
\[\MBR{\LG}=\mathcal{H}_\ell^\sigma(\LG)\cong \mathcal{LM}(\LL)\cong \MG\cong \mathcal{RM}(\LL)\cong\mathcal{H}_r^\sigma(\LG)=\MBL{\LG}.\]
\end{cor}
\begin{proof}
Let $\phi$ be a bounded normal linear mapping on $\LG$, then $\omega\cdot\phi(x)=(\iota\otimes\omega)\Gamma(\phi(x))$ and
$$\phi(\omega\cdot x)=\phi((\iota\otimes\omega)\Gamma x)
=(\iota\otimes\omega)((\phi\otimes\iota)\Gamma x).$$
By density and continuity we obtain $\MBR{\LG}=\mathcal{H}_\ell^\sigma(\LG)$. The other one can be proved similarly.
\end{proof}

\section{Embedding of $\MG$ into $\LUC^*$ and the topological centre of $\LUC^*$}\label{embed}
If $G$ is a locally compact group, then ${\rm M}(G)$ can  be isometrically embedded in ${\rm LUC}(G)^*$ by $\mu\mapsto\hat{\mu}$,~~where $\hat{\mu}(f)=\int f d\mu$. More precisely, Lau \cite[Theorem 1]{L_cont} showed that ${\rm M}(G)$ is isometrically isomorphic to $\mathcal Z({\rm LUC}(G)^*)$. On the other hand, Lau and Losert \cite[Section 4]{L-LOS} established such an embedding from ${\rm B}_\rho(G)$ into $\mathcal Z({\rm UCB}(\hat{G})^*)$.  Here we show that these both results are special case of a more general fact  about locally compact quantum groups.

For a co-amenable locally compact quantum group $\G$,  let $\theta_1: \MG\rightarrow \mathcal{H}_r^\sigma(\LG)$ and $\theta_2: \mathcal{H}_r(\LG)\rightarrow \LUC^*$ be the isometric isomorphisms established  in Corollary \ref{MBR} and Theorem \ref{LG}, respectively. Then
$$\Theta:=\theta_2\circ i\circ\theta_1: \MG\rightarrow \LUC^*$$
defines an isometric algebra homomorphism, where $i: \mathcal{H}_r^\sigma(\LG)\rightarrow \mathcal{H}_r(\LG)$ is the inclusion mapping.

We have the following lemma which can be regarded as an extension of \cite[Proposition 4.2]{L-LOS},

\begin{lem}\label{X=MG}
Let $\G$ be a locally compact quantum group, then
\bit
\item[(i)] there exists  an isometric algebra homomorphism $\Pi:M(\G)\to \LUC^*$;
\item[(ii)] $\Pi=\Theta$, provided that $\G$ is co-amenable.
\eit
\end{lem}

\begin{proof}
Given $m\in M(\G)$, since $\CL$ is a C$^*$-algebra, $m$ has a unique extension $\tilde{m}$ to the multiplier algebra $\mathcal{M}(\CL)$ of $\CL$. Now let $\Pi(m)$ be the restriction of $\tilde{m}$ to $\LUC$. Since $\CL\subseteq\LUC\subseteq\mathcal{M}(\CL)$,~\cite[Theorem 2.3]{R_uniform}, $\Pi(m)$ is the unique norm preserving extension of $m$ to $\LUC$, so that $\Pi$ is a linear isometry from $M(\G)$ into $\LUC^*$.
For $m,n\in M(\G)$, both $\Pi(m\ast n)$ and $\Pi(m)\odot\Pi(n)$ are norm preserving extension of $m\ast n$ from $\CL$ to $\LUC$  and
 $\Pi(m\ast n)=\Pi(m)\odot\Pi(n)$. Thus $\Pi$ is a homomorphism.

Now assume that $\G$  is co-amenable.  Let  $y\in \CL$, then  by the Cohen Factorization Theorem \cite[Corollary 2.9.26]{Dal00} there exist $x\in\CL$ and $\omega\in\LL$ such that $y=(\omega\otimes\iota)\Gamma_cx.$ Then
$$\Theta(m)(y)=x(\omega*m)=m((\omega\otimes\iota)\Gamma_cx)=m(y).$$
Hence $\Theta(m)$ is a norm preserving extension of $m$ to $\LUC$. Thus the uniqueness implies that $\Pi(m)=\Theta(m)$ for each $m\in\MG$.
\end{proof}

Part (ii) of the next result is an extension of \cite[Theorem 4.12]{L-LOS}.

\begin{prop}\label{ZLUCG}
Let $\G$ be a co-amenable locally compact quantum group, then
\bit
\item[(i)] $\Theta(\MG)\subseteq \mathcal{Z}(\LUC^*)$;
\item[(ii)] $\Theta$ is surjective if and only if $\G$ is compact.
\eit
\end{prop}

\begin{proof}
It is enough to show that for each $m\in \MG$ the mapping $n\to \Theta(m)\odot n$ from $\LUC^*$ into itself is normal. To see this, let $\{n_\alpha\}$ be a net in $\LUC^*$ tending to $n$ in the w*-topology and let $x\in \LUC$. Thus there exist $\omega\in\LL$ and  $ y\in\LG$  such that  $ x=y\cdot\omega$,  so
\begin{eqnarray*}
(\Theta(m)\odot n_\alpha)(x)&=&\Theta(m)(n_\alpha\odot x)=\Theta(m)((n_\alpha\odot y)\cdot \omega)\\
&=&(n_\alpha\odot y)(\omega\ast m)=n_\alpha(y\cdot(\omega\ast m))\\
&\to&n(y\cdot(\omega\ast m))=(\Theta(m)\odot n)(x).
\end{eqnarray*}
(ii) It is clear that $\Theta$ is surjective if and only if the inclusion mapping $i: \mathcal{H}_r^\sigma(\LG)\rightarrow \mathcal{H}_r(\LG)$  is surjective. As Theorem~\ref{Comp} shows, the latter fact is equivalent to the compactness of $\G$.
\end{proof}

\begin{cor}\label{LUC=WAP}
Let $\G$ be a co-amenable locally compact quantum group, then the following assertions are equivalent:
\bit
\item[(i)] $\G$ is compact;
\item[(ii)] $\LUC={\rm WAP}(\G)$ and $\Theta(\MG)=\mathcal{Z}(\LUC^*)$.
\eit
\end{cor}

\begin{proof}
 For (i)$\Rightarrow $(ii), it is easy to verify that if $\G$ is compact then $\LUC={\rm WAP}(\G)$ (see \cite[Theorem 4.3]{R_uniform}).  The equality $\Theta(\MG)=\mathcal{Z}(\LUC^*)$ trivially follows from the fact that $\Theta$ is surjective. If (ii) holds then $\Theta(\MG)=\mathcal{Z}(\LUC^*)=\mathcal{Z}({\rm WAP}(\G)^*)={\rm WAP}(\G)^*=\LUC^*$; that is, $\Theta$ is surjective and so $\G$ is compact by  Proposition  \ref{ZLUCG}.
\end{proof}

For a co-amenable locally compact quantum group $\G$, if $\Theta(\MG)=\mathcal{Z}(\LUC^*)$ then the latter result confirms that $\LUC={\rm WAP}(\G)$ if and only if $\G$ is compact. This provides a positive answer to a problem raised by Runde; see the remark just after \cite[Theorem 4.3]{R_uniform}.

For the classical case $\G_a$, as it has been shown by Lau in \cite[Theorem 1]{L_cont}, we always have the equality $\Theta({\rm M}(\G_a))= \mathcal{Z}({\rm LUC}(\G_a)^*)$. Therefore, ${\rm LUC}(\G_a)={\rm WAP}(\G_a)$ if and only if $\G_a$ is compact; see \cite[Corollary 4]{L_cont}.  In contrast to the case $\G_a$, the situation for the case  $\G_s$ is slightly  different. Viktor Losert, surprisingly pointed out that the equality  $\Theta({\rm M}(\G_s))= \mathcal{Z}({\rm LUC}(\G_s)^*)$ does not valid, in general. Indeed, he showed that $\mathcal{Z}({\rm UCB}(\hat{SU_3})^*)\neq {\rm B}(SU_3).$

We conclude our work with an application of Corollary~\ref{LUC=WAP}, for the classical case $\G_s$. Note that in all case presented in the next result we have the equality $\mathcal{Z}({\rm UCB}(\hat{G})^*)={\rm B}_\rho(G)$, see \cite[Section 3]{HU}.

 \begin{cor}
Let $G$ be a locally compact group enjoying either  of the following properties:
\bit
\item[(i)] $G$ is a metrizable group and  $\overline{[G,G]}$ is not open in $G$, where $[G,G]$ denote the commutator subgroup of $G$;
\item[(ii)] $G=\prod_iG_i$ is a finite or countable product of metrizable locally compact groups such that $G_i$ is compact for all but finitely many $i$ and either $\overline{[G_1,G_1]}$ is not open in $G_1$ or $G_1$ is abelian and non-discrete;
\item[(iii)] $G=G_0\times \prod_{i=1}^\infty G_i$, where each $G_i~(i\geq 0)$ is a metrizable locally compact group and $G_i$ is compact and non-trivial for $i\geq 1$.
\eit
Then   ${\rm UCB}(\hat{G})={\rm W}(\hat{G})$ if and only if $G$ is discrete.
\end{cor}
\section{addendum}
After the corrected galleyproofs of paper being appeared online (in JMAA), Professor
Matthew Daws brought to our attention that the proof of part (i) of
Lemma 4.1 needs more details. With special thanks to him, we
present a detailed proof as follows:\\

Let $m\in\MG$ and $(e_\alpha)$ be a BAI for $\CL$. By [K, section 7] there exists $\tilde{m}\in\mathcal{M}(\CL)^*$ such that
$$\tilde{m}(x)=\lim_\alpha m(xe_\alpha)\fs\fs(x\in\mathcal{M}(\CL)).$$
It is proved  that $\tilde{m}$ is a well defined unique extension of $m$ to $\mathcal{M}(\CL)$ with the same norm.\\

For $m\in\MG$ there exists $m'\in\MG$ and $b\in\CL$ such that $m=m'b$. This implies  that $$\tilde{m}(x)=m'(bx)\fs\fs (x\in\mathcal{M}(\CL)).$$

Let $m,n\in\MG$. First note that $\tilde{m}\odot\tilde{n}$ is an extension of $m\ast n$ to $\mathcal{M}(\CL)$. Indeed, if $a\in\CL$ then $\tilde{n}\odot a=n\odot a$. So by Proposition 2.1  we have
\beq*
(m\ast n)(a)&=&m(\underbrace{n\odot a}_{\in\CL})=\tilde{m}(n\odot a)
=\tilde{m}\big{(}(\tilde{n}\odot a)\big{)}
=(\tilde{m}\odot\tilde{n})(a).
\eeq*

Now we show that $\tilde{m}\odot\tilde{n}=\widetilde{m\ast n}$. Since $\Gamma_c:\CL\to\mathcal{M}(\CL\otimes\CL)$ is a non-degenerate $\ast$-homomorphism, $\tilde{\Gamma}_c$ is strictly continuous on bounded subsets of $\mathcal{M}(\CL)$, [K, Proposition 7.2]. But for each $x\in\mathcal{M}(\CL)$ we have
$$(\tilde{m}\odot x)a=(\iota\otimes m')(\tilde{\Gamma}_c x(a\otimes b)),$$
where $m'\in\MG$ and $b\in\CL$ is such that $m=m'b$. Thus $\tilde{m}\odot x_i$  converges strictly to $\tilde{m}\odot x$, for every bounded net $(x_i)$ in $\mathcal{M}(\CL)$ that is strictly converges to some $x\in\mathcal{M}(\CL)$.

Let $x\in\mathcal{M}(\CL)$. It is clear that $(xe_\alpha)$ is a bounded net in $\CL$ which converges strictly to $x$. So
$\tilde{n}\odot (xe_\alpha)$ converges strictly to $\tilde{n}\odot x$ and by strict continuity of $\tilde{n}$ on bounded subsets of $\mathcal{M}(\CL)$, [K, Proposition 7.2] we have
\beq*
(\tilde{m}\odot\tilde{n})(x)=\lim_\alpha\tilde{m}(\tilde{n}\odot(xe_\alpha))
=\lim_\alpha\tilde{m}\odot\tilde{n}(\underbrace{xe_\alpha}_{\in\CL})
=\lim_\alpha(m\ast n)(xe_\alpha)
=\widetilde{m\ast n}(x).
\eeq*
Now if we define  $\prod(m)$ as the restriction of  $\tilde{m}$ to $LUC(\G)$ then $\prod$ has the desired properties.\\

\noindent [K]   {J. Kustermans, }{\it One-parameter representations on C$^*$-algebras,} Preprint, Odense Universitet, 1997.
\#funct-an/9707010.

\bibliographystyle{alpha}

\end{document}